\documentclass[12pt]{article}
\usepackage[utf8]{inputenc}
\usepackage{amsmath}
\usepackage{amsthm}
\usepackage{amssymb}
\usepackage{cite}
\usepackage{multicol,url}
\usepackage{tikz-cd}
\usepackage{enumitem}
\usepackage{mathrsfs}
\usepackage{amsfonts}

\newtheorem{theorem}{Theorem}[section]
\newtheorem{corollary}[theorem]{Corollary}
\newtheorem{lemma}[theorem]{Lemma}
\newtheorem{proposition}[theorem]{Proposition}
\newtheorem{definition}[theorem]{Definition}
\newtheorem{problem}{Problem}
\newtheorem*{claim}{Claim}
\newtheorem*{remark}{Remark}
\newtheorem*{note}{Note}

\renewcommand{\and}{\text{ and }}

\newcommand{\BC}{\begin{center}}
\newcommand{\EC}{\end{center}}
\newcommand{\BCS}{\begin{cases}}
\newcommand{\ECS}{\end{cases}}
\newcommand{\BM}{\begin{pmatrix}}
\newcommand{\EM}{\end{pmatrix}}
\newcommand{\BD}{\begin{tikzcd}}
\newcommand{\ED}{\end{tikzcd}}
\newcommand{\BE}{\begin{equation}}
\newcommand{\EE}{\end{equation}}
\newcommand{\thm}{\begin{theorem}}
\newcommand{\ethm}{\end{theorem}}
\newcommand{\prf}{\begin{proof}}
\newcommand{\eprf}{\end{proof}}
\newcommand{\cor}{\begin{corollary}}
\newcommand{\ecor}{\end{corollary}}
\newcommand{\prop}{\begin{proposition}}
\newcommand{\eprop}{\end{proposition}}
\newcommand{\prob}[1]{\begin{problem}[#1]}
\newcommand{\eprob}{\end{problem}}
\newcommand{\lem}{\begin{lemma}} 
\newcommand{\elem}{\end{lemma}}
\newcommand{\defi}{\begin{definition}}
\newcommand{\edefi}{\end{definition}}
\newcommand{\clm}{\begin{claim}}
\newcommand{\eclm}{\end{claim}}
\newcommand{\rem}{\begin{remark}}
\newcommand{\erem}{\end{remark}}
\newcommand{\nt}{\begin{note}}
\newcommand{\ent}{\end{note}}
\newcommand{\renumerate}{\begin{enumerate}[label=(\roman*)]}
\newcommand{\eenumerate}{\end{enumerate}}
\newcommand{\lenumerate}{\begin{enumerate}[label=(\alph*)]}
\newcommand{\nenumerate}{\begin{enumerate}[label=(\arabic*)]}

\begin{document}

\title{\bf
The $q$-Onsager algebra and the \\ quantum torus
}
\author{
Owen Goff}
\date{}
\maketitle
\begin{abstract} The $q$-Onsager algebra, denoted $O_q$, is defined by two generators $W_0, W_1$ and two relations called the $q$-Dolan-Grady relations. 
Recently, Terwilliger introduced some elements of $O_q$, said to be alternating. These elements are denoted
\begin{equation} \label{intrp332}
 \{{W}_{-k}\}_{k=0}^{\infty}, \qquad
\{{W}_{k+1}\}_{k=0}^{\infty}, \qquad  \{{G}_{k+1}\}_{k=0}^{\infty}, \qquad \{{\tilde{G}}_{k+1}\}_{k=0}^{\infty}.  \nonumber
\end{equation}

The alternating elements of $O_q$ are defined recursively. By construction, they are polynomials in $W_0$ and $W_1$. It is currently unknown how to express these polynomials in closed form.

\medskip
In this paper, we consider an algebra $T_q$, called the \text{quantum torus}. We present a basis for $T_q$ and define an algebra homomorphism $p: O_q \mapsto T_q$. In our main result, we express the $p$-images of the alternating elements of $O_q$ in the basis for $T_q$. These expressions are in a closed form that we find attractive.
\bigskip

\noindent
{\bf Keywords}. $q$-Onsager algebra; $q$-Dolan-Grady relations; quantum torus; alternating central extension.
\hfil\break
\noindent {\bf 2020 Mathematics Subject Classification}.
Primary: 17B37.
Secondary: 05E10, 16T20.
 \end{abstract}

\section{Introduction}

  The $q$-Onsager algebra, which we denote $O_q$, first appeared in the topic of algebraic graph theory \cite[Lemma 5.4]{assocscheme}. Applications of $O_q$ in algebraic combinatorics appear in many papers, including \cite{xxy, IAT, 1865045}. The algebra $O_q$ has recently been applied to statistical mechanics \cite{BB, in, Source16, ddg, XX} 
 and quantum groups \cite[Section 2.2]{p}.  The algebra $O_q$ is defined by generators $W_0, W_1$ and two relations called the $q$-Dolan-Grady relations.

  \medskip 
  
  We mentioned that $O_q$ appears in statistical mechanics; let us expand on that. In \cite[eqn. 4]{in}, Pascal Baseilhac and Kozo Koizumi described an algebra similar to $O_q$ while researching operators on boundary integrable systems with
 hidden symmetries. Following \cite{T2}, we denote this algebra by $\mathcal O_q$. This algebra is defined by generators and relations. The generators are \begin{equation} \label{intrp}
 \{\mathcal{W}_{-k}\}_{k=0}^{\infty}, \qquad
\{\mathcal{W}_{k+1}\}_{k=0}^{\infty}, \qquad  \{\mathcal{G}_{k+1}\}_{k=0}^{\infty}, \qquad \{\mathcal{\tilde{G}}_{k+1}\}_{k=0}^{\infty}. 
\end{equation}
The relations are given in \cite[Definition 3.1]{Source16}, as well as Definition \ref{note} below. The name and notation for $\mathcal O_q$ have evolved over time; it is currently known as the alternating central extension of the $q$-Onsager algebra, as in \cite{T3}. The elements displayed in equation \eqref{intrp} are called the alternating generators of $\mathcal O_q$; see \cite{T3} and \cite{T2}.

\medskip 

Motivated by \cite{in}, Paul Terwilliger showed in \cite{T3} that there exists an algebra isomorphism from $\mathcal O_q$ to the tensor product of $O_q$ and the polynomial algebra over countably many commuting variables. We denote this isomorphism by $\varphi$, and we denote the variables by $\{z_i\}_{i=1}^{\infty}$. For $i \ge 1$, let $\mathcal Z_i \in \mathcal O_q$ denote the $\varphi$-preimage of $1 \otimes z_i$. By \cite[Definition 11.2]{T2}, there exists an algebra homomorphism $\gamma: \mathcal O_q \mapsto O_q$ that sends $\mathcal W_0 \mapsto W_0$, $\mathcal W_1 \mapsto W_1$, and $\mathcal Z_i \mapsto 0$ for $i \ge 1$. 

\medskip

The $\gamma$-images of the alternating generators of $\mathcal O_q$ are called the alternating elements of $O_q$. By construction, each alternating element of $O_q$ is a polynomial in $W_0$ and $W_1$. Recursive formulas for these polynomials are known; see \cite[Lemma 8.22]{T3}. However, closed forms for the polynomials are not currently known. 
\medskip 

In this paper, we discuss an algebra $T_q$, called the {quantum torus} \cite{Magic}. The algebra $T_q$ is defined by the generators $x$, $y$, $x^{-1}$, $y^{-1}$ and the relations $$xx^{-1} = 1 = x^{-1}x, \qquad yy^{-1} = 1 = y^{-1}y, \qquad xy=q^2yx.$$ It is known that the elements $\{x^iy^j\vert i, j \in \mathbb{Z}\}$ form a basis for $T_q$; see \cite[p. 3]{Magic}.
\medskip

In this paper, we introduce the algebra homomorphism $p: O_q \rightarrow T_q$ that sends $W_0 \mapsto x+x^{-1}$ and $W_1 \mapsto y+y^{-1}$. We consider the $p$-images of the alternating elements of $O_q$. Our main objective is to express these $p$-images in closed form in the above basis. To reach our objective, we make use of generating functions. We give two versions of our main result: one expressed using generating functions and one without them. The main results of this paper are Theorem \ref{hartshorne} and Theorem \ref{halfmarathon}.

\medskip 

The paper is organized as follows. In Sections 3--5, we give definitions and results concerning $\mathcal O_q$, $O_q$, and $T_q$. Section 6 covers generating functions, and in Section 7 we provide notation and formulas. In Section 8 we state and prove our main result in terms of generating functions. In Section 9 we present our main result without using generating functions.

\section{Notations}

In this section, we introduce some notation that will be used throughout the paper.
\medskip

\begin{itemize}

\item Let ${K}$ be a field of characteristic $0$.

\item All algebras in this paper are associative, unital, and over $K$.

\item For $u,v$ in any algebra, define $[u,v]=uv-vu$. We call $[u,v]$ the \textit{commutator} of $u$ and $v$. For nonzero $r \in K$, define $[u,v]_r=ruv-r^{-1}vu$. We call $[u,v]_r$ the \textit{r-commutator} of $u$ and $v$.

\item Fix nonzero $q \in K$.



\item Let $\mathbb{N}$ denote the set of natural numbers $\{0, 1, 2, \ldots \}$. Let $\mathbb{Z}$ denote the set of integers. Let $\mathbb{C}$ denote the set of complex numbers.

\item For $n\in \mathbb{N}$, define $[n]_q=\frac{q^n-q^{-n}}{q-q^{-1}}$. We are using the notation from \cite{T1}.

\item By an \textit{automorphism} of an algebra $\mathcal A$, we mean an algebra isomorphism from $\mathcal A$ to itself.

\item The \textit{opposite algebra} of an algebra $\mathcal A$, denoted $\mathcal A^{op}$, is the algebra on the vector space $\mathcal A$ such that for $a,b \in \mathcal A^{op}$, $ab$ (in $\mathcal A^{op}$) $=ba$ (in $\mathcal A$). Note that if $\mathcal A$ is commutative, then $\mathcal A = \mathcal A^{op}$.

\item By an \textit{antiautomorphism} of an algebra $\mathcal{A}$, we mean an algebra isomorphism from $\mathcal A$ to $\mathcal A^{op}$. In other words, an antiautomorphism of $\mathcal A$ is a $K$-linear bijection $\alpha:\mathcal{A} \rightarrow \mathcal{A}$ that reverses the order of multiplication: $\alpha(ab) = \alpha(b)\alpha(a)$ for all $a,b \in \mathcal A$.

\end{itemize}
\medskip






\section{The $q$-Onsager algebra}

In this section, we consider the $q$-Onsager algebra and its alternating central extension.

\begin{definition}
\rm
Let $O_q$ denote the algebra defined by generators $W_0$, $W_1$ and relations
\begin{align}
[{W}_0,[{W}_0,[{W}_0,{W}_1]_q]_{q^{-1}}] &= -(q^2-q^{-2})^2[{W}_0,{W}_1], \label{eq:2p1a} \\ 
[{W}_1,[{W}_1,[{W}_1,{W}_0]_q]_{q^{-1}}] &= -(q^2-q^{-2})^2[{W}_1,{W}_0].
\label{eq:2p2a} 
\end{align}
We call $O_q$ the $q$-\textit{Onsager algebra}.
\end{definition}

\medskip

Equations \eqref{eq:2p1a} and \eqref{eq:2p2a} are called the $q$-\textit{Dolan-Grady relations.} We define an automorphism and an antiautomorphism of $O_q$.

\begin{lemma}  {\rm (See \cite[Lemma 3.3]{T2}.)} \label{molybdenum}
There exists an automorphism $\sigma$ of $O_q$ that sends
$$W_0 \mapsto W_1, \qquad \qquad W_1 \mapsto W_0.$$

\end{lemma}

\begin{lemma} {\rm (See \cite[Lemma 3.4]{T2}.)} \label{technetium}
 There exists an antiautomorphism $\dagger$ of $O_q$ that sends
$$W_0 \mapsto W_0, \qquad \qquad W_1 \mapsto W_1.$$
\end{lemma}

\begin{lemma}
With reference to Lemmas \ref{molybdenum} and \ref{technetium}, we have $\sigma \circ \dagger = \dagger \circ \sigma$ in $O_q$. That is to say, this diagram commutes:
\begin{center}
\begin{tikzcd}
O_q \arrow[r,"\dagger"] \arrow[d,"\sigma"'] & O_q \arrow[d,"\sigma"] \\
O_q \arrow[r,"\dagger"'] &  O_q 
\end{tikzcd}
\end{center}
\end{lemma}

\begin{proof}
Chase $W_0$ and $W_1$ around the diagram.
\end{proof}

We recall the \text{alternating central extension} of $O_q$.


\begin{definition}\label{note}{\rm (See \cite[Definition 3.1]{Source16})} \rm  Define the algebra $\mathcal{O}_q$ by the generators 
\begin{equation} \label{e}
 \{\mathcal{W}_{-k}\}_{k=0}^{\infty}, \qquad
\{\mathcal{W}_{k+1}\}_{k=0}^{\infty}, \qquad  \{\mathcal{G}_{k+1}\}_{k=0}^{\infty}, \qquad \{\mathcal{\tilde{G}}_{k+1}\}_{k=0}^{\infty}, 
\end{equation}
and the following relations.  For $k,\ell \in \mathbb{N}$,
\begin{align}
&
 \lbrack \mathcal W_0, \mathcal W_{k+1}\rbrack= 
\lbrack \mathcal W_{-k}, \mathcal W_{1}\rbrack=
({\mathcal{\tilde G}}_{k+1} - \mathcal G_{k+1})/(q+q^{-1}),
\label{eq:3p1}
\\
&
\lbrack \mathcal W_0, \mathcal G_{k+1}\rbrack_q= 
\lbrack {\mathcal{\tilde G}}_{k+1}, \mathcal W_{0}\rbrack_q= 
\rho  \mathcal W_{-k-1}-\rho 
 \mathcal W_{k+1},
\label{eq:3p2}
\\
&
\lbrack \mathcal G_{k+1}, \mathcal W_{1}\rbrack_q= 
\lbrack \mathcal W_{1}, {\mathcal {\tilde G}}_{k+1}\rbrack_q= 
\rho  \mathcal W_{k+2}-\rho 
 \mathcal W_{-k},
\label{eq:3p3}
\\
&
\lbrack \mathcal W_{-k}, \mathcal W_{-\ell}\rbrack=0,
\label{eq:3p31}
\\
&
\lbrack \mathcal W_{k+1}, \mathcal W_{\ell+1}\rbrack= 0,
\label{eq:3p4}
\\
&
\lbrack \mathcal W_{-k}, \mathcal W_{\ell+1}\rbrack+
\lbrack \mathcal W_{k+1}, \mathcal W_{-\ell}\rbrack= 0,
\label{eq:3p5}
\\
&
\lbrack \mathcal W_{-k}, \mathcal G_{\ell+1}\rbrack+
\lbrack \mathcal G_{k+1}, \mathcal W_{-\ell}\rbrack= 0,
\label{eq:3p6}
\\
&
\lbrack \mathcal W_{-k}, {\mathcal {\tilde G}}_{\ell+1}\rbrack+
\lbrack {\mathcal {\tilde G}}_{k+1}, \mathcal W_{-\ell}\rbrack= 0,
\label{eq:3p7}
\\
&
\lbrack \mathcal W_{k+1}, \mathcal G_{\ell+1}\rbrack+
\lbrack \mathcal  G_{k+1}, \mathcal W_{\ell+1}\rbrack= 0,
\label{eq:3p8}
\\
&
\lbrack \mathcal W_{k+1}, {\mathcal {\tilde G}}_{\ell+1}\rbrack+
\lbrack {\mathcal {\tilde G}}_{k+1}, \mathcal W_{\ell+1}\rbrack= 0,
\label{eq:3p9}
\\
&
\lbrack \mathcal G_{k+1}, \mathcal G_{\ell+1}\rbrack=0,
\label{eq:3p32}
\\
&
\lbrack {\mathcal {\tilde G}}_{k+1}, {\mathcal {\tilde G}}_{\ell+1}\rbrack= 0,
\label{eq:3p10}
\\
&
\lbrack {\mathcal {\tilde G}}_{k+1}, \mathcal G_{\ell+1}\rbrack+
\lbrack \mathcal G_{k+1}, {\mathcal {\tilde G}}_{\ell+1}\rbrack= 0.
\label{eq:3p11}
\end{align}

In these equations, $\rho = -(q^2-q^{-2})^2$.

\medskip
The algebra $\mathcal O_q$ is called the \textit{alternating central extension} of $O_q$. The elements in \eqref{e} are called the \textit{alternating generators} of $\mathcal{O}_q$.
\end{definition}
For notational convenience, we define 
$$\mathcal{G}_0=-(q-q^{-1})(q+q^{-1})^2, \qquad \tilde{\mathcal{G}}_0=-(q-q^{-1})(q+q^{-1})^2.$$ 

For the remainder of this section, we explain how $O_q$ and $\mathcal{O}_q$ are related.

\begin{proposition}{\rm (See \cite[Section~3]{BB}.)}
In $\mathcal{O}_q$, we have
\begin{align}
[\mathcal{W}_0,[\mathcal{W}_0,[\mathcal{W}_0,\mathcal{W}_1]_q]_{q^{-1}}] &= -(q^2-q^{-2})^2[\mathcal{W}_0,\mathcal{W}_1], \nonumber \\
[{\mathcal W}_1,[{\mathcal W}_1,[{\mathcal W}_1,{\mathcal W}_0]_q]_{q^{-1}}] &= -(q^2-q^{-2})^2[{\mathcal W}_1,{\mathcal W}_0]. \nonumber 
\end{align}
\end{proposition}

We have defined an automorphism and an antiautomorphism of $O_q$. We define analogous maps for $\mathcal O_q$.
\begin{lemma}{\rm (See \cite[Lemma 3.1]{T3}.)}
There exists an automorphism $\sigma$ of $\mathcal O_q$ that sends
$$ \mathcal W_{-k} \mapsto \mathcal W_{k+1} , \qquad \mathcal W_{k+1} \mapsto \mathcal W_{-k}, \qquad \mathcal G_{k+1} \mapsto \tilde{\mathcal G}_{k+1}, \qquad \tilde{\mathcal G}_{k+1} \mapsto \mathcal G_{k+1}$$
for all $k \in \mathbb{N}$.
\end{lemma}

\begin{lemma}
 {\rm(See \cite[Lemma 3.2]{T3}.)}
There exists an antiautomorphism $\dagger$ of $\mathcal O_q$ that sends
$$ \mathcal W_{-k} \mapsto \mathcal W_{-k} , \qquad \mathcal W_{k+1} \mapsto \mathcal W_{k+1} ,\qquad  \mathcal G_{k+1} \mapsto \tilde{ \mathcal G}_{k+1} , \qquad \tilde{\mathcal G}_{k+1} \mapsto \mathcal G_{k+1}$$
for all $k \in \mathbb{N}$.
\end{lemma}

\begin{lemma} \label{westernmassachusetts}
This diagram commutes:
\begin{center}
\begin{tikzcd}
\mathcal O_q \arrow[r,"\dagger"] \arrow[d,"\sigma"'] & \mathcal O_q \arrow[d,"\sigma"] \\
\mathcal O_q \arrow[r,"\dagger"'] & \mathcal O_q 
\end{tikzcd}
\end{center}
\end{lemma}

\begin{proof}
Chase the alternating generators of $\mathcal O_q$ around the diagram.
\end{proof}


We bring in some notation. Let $\{z_i\}_{i = 1}^{\infty}$ be mutually commuting indeterminates. Let $K[z_1,z_2, \ldots ]$ denote the algebra of polynomials in $\{z_i\}_{i=1}^{\infty}$ that have coefficients in $K$. 

\begin{proposition} \label{ruthenium}
There exists an algebra isomorphism
$$\phi: O_q \otimes K[z_1,z_2, \ldots ] \mapsto \mathcal{O}_q$$ such that $$\phi(W_0 \otimes 1) = \mathcal{W}_0, \qquad \qquad \phi(W_1 \otimes 1) = \mathcal W_1,$$ and these diagrams commute:

\begin{center}
\begin{tikzcd}
O_q \otimes K[z_1,z_2,\ldots] \arrow[r,"\phi"] \arrow[d,"\sigma \otimes id"'] & \mathcal O_q \arrow[d,"\sigma"] \\
O_q \otimes K[z_1,z_2,\ldots] \arrow[r,"\phi"'] & \mathcal O_q
\end{tikzcd}
\qquad
\begin{tikzcd}
O_q \otimes K[z_1,z_2,\ldots] \arrow[r,"\phi"] \arrow[d,"\dagger \otimes id"'] & \mathcal O_q \arrow[d,"\dagger"] \\
O_q \otimes K[z_1,z_2,\ldots] \arrow[r,"\phi"'] & \mathcal O_q.
\end{tikzcd}
\end{center}
\end{proposition}
\begin{proof}
The map $\phi$ is the isomorphism defined in \cite[Proposition 8.8]{T2}. The diagrams commute by \cite[Proposition 8.9]{T2}.
\end{proof}

\begin{definition} \label{pomegranate}
\rm
For $i\geq 1$, define $\mathcal Z_i = \phi(1 \otimes z_i)$. For notational convenience, define $\mathcal Z_0 = (q+q^{-1})^2$. \end{definition}

\begin{lemma}
The elements $$\mathcal W_0, \qquad \mathcal W_1, \qquad \{\mathcal Z_i\}_{i=1}^{\infty}$$ generate $\mathcal O_q$. Moreover, the elements $\{\mathcal Z_i\}_{i=1}^{\infty}$ are central in $\mathcal O_q$.
\end{lemma}
\begin{proof}
By Proposition \ref{ruthenium} and since $1 \otimes z_i$ is central in $O_q \otimes K[z_1,z_2,\ldots]$ for $i \ge 1.$
\end{proof}

\begin{definition} \label{newlyadded} \rm
Define an algebra homomorphism $\epsilon : K[z_1,z_2,\ldots] \mapsto K$ that sends $z_i \mapsto 0$ for $i \ge 1.$
\end{definition}

\begin{definition} \rm{(See \cite[Definition 11.2]{T2}.)} \label{indium}
\rm Let $\gamma: \mathcal O_q \mapsto O_q$ denote the composition
$$\gamma: \quad \mathcal{O}_q \xrightarrow{\phi^{-1}} O_q \otimes K[z_1,z_2,\ldots] \xrightarrow{id \otimes \epsilon} O_q \otimes K \xrightarrow{x \otimes 1 \mapsto x} O_q,$$
where $\phi$ is from Proposition \ref{ruthenium}.
\end{definition}
\medskip
Note that $\gamma$ is an algebra homomorphism.

\begin{lemma} \label{pittsburgh}
The map $\gamma$ sends
\begin{align*}
\mathcal W_0 \mapsto W_0, \qquad \mathcal W_1 \mapsto W_1, \qquad
\mathcal Z_i \mapsto 0, \quad i\geq 1.
\end{align*}
\end{lemma}

\begin{proof}
Routine consequence of Definition \ref{indium}.
\end{proof}

\begin{definition} \label{worc}
\rm (See \cite[Definition 11.5]{T2}.)
For $k \in \mathbb{N}$, define
$$W_{-k} = \gamma(\mathcal W_{-k}), \qquad W_{k+1} = \gamma(\mathcal W_{k+1}) , \qquad G_{k+1} = \gamma(\mathcal G_{k+1}), \qquad \tilde G_{k+1} = \gamma(\mathcal{\tilde{G}}_{k+1}).$$
We call these elements the \textit{alternating elements} of $O_q$.
For notational convenience, we define
\begin{align}
     G_0 = -\left(q-q^{-1}\right)\left(q+q^{-1}\right)^2, \qquad \qquad
    \tilde{G}_0  =  -\left(q-q^{-1}\right)\left(q+q^{-1}\right)^2.  \nonumber
\end{align}
\end{definition}

\begin{lemma} \label{imani}
These diagrams commute:
\begin{center}
\begin{tikzcd}
\mathcal O_q \arrow[r,"\sigma"] \arrow[d,"\gamma"'] & \mathcal O_q \arrow[d,"\gamma"] \\
O_q \arrow[r,"\sigma"'] & O_q 
\end{tikzcd}
\qquad \qquad
\begin{tikzcd}
\mathcal O_q \arrow[r,"\dagger"] \arrow[d,"\gamma"'] & \mathcal O_q \arrow[d,"\gamma"] \\
O_q \arrow[r,"\dagger"'] & O_q 
\end{tikzcd}
\end{center}
\end{lemma}

\begin{proof}
Chase $\mathcal{W}_0$, $\mathcal{W}_1$, and $\mathcal Z_i$ around each diagram.
\end{proof}

\begin{corollary}
Pick $k \in \mathbb{N}.$
The automorphism $\sigma$ of $O_q$ sends
$$W_{-k} \mapsto  W_{k+1} , \qquad  W_{k+1} \mapsto  W_{-k}, \qquad G_{k+1} \mapsto \tilde{ G}_{k+1}, \qquad \tilde{ G}_{k+1} \mapsto  G_{k+1}.$$
The antiautomorphism $\dagger$ of $O_q$ sends
$$W_{-k} \mapsto W_{-k} , \qquad   W_{k+1} \mapsto W_{k+1} ,\qquad  G_{k+1} \mapsto \tilde{ G}_{k+1} , \qquad \tilde{G}_{k+1} \mapsto  G_{k+1}.$$

\end{corollary}

\begin{proof}
Chase $\mathcal W_{-k}$, $\mathcal W_{k+1}$,  $\mathcal G_{k+1}$, and $\tilde{\mathcal G}_{k+1}$ around the diagrams in Lemma \ref{imani}.
\end{proof}



\section{The quantum torus}

In this section we consider an algebra called the \text{quantum torus}, denoted $T_q$. We review some properties of $T_q$ and display a basis for the vector space $T_q$. Then we introduce an algebra homomorphism $p: O_q \mapsto T_q$ and consider the $p$-images of the alternating elements of $O_q$.

\begin{definition} 
\rm (See \cite{Magic}.)
Define the algebra $T_q$ by generators
$$x,y,x^{-1},y^{-1}$$
and relations
\begin{align}
xx^{-1}=1=x^{-1}x, \qquad \qquad
yy^{-1} = 1 = y^{-1}y, \qquad \qquad xy=q^2yx. \label{aldo}
\end{align}

The algebra $T_q$ is called the \textit{quantum torus}. 
\end{definition}

\begin{lemma} \label{lem:3facts}
The following relations hold in $T_q$:
\begin{align}
xy &= q^2yx, \nonumber \\
x^{-1}y &= q^{-2}yx^{-1}, \nonumber\\
x^{-1}y^{-1} &= q^2y^{-1}x^{-1}, \nonumber \\
xy^{-1} &= q^{-2}y^{-1}x. \label{eq:8p4} \nonumber
\end{align}
\end{lemma}

\begin{proof}
Routine consequence of the relations in \eqref{aldo}.
\end{proof}







We display a basis for the vector space $T_q$.

\begin{lemma}{\rm (See \cite[p.~3]{Magic}.)} \label{iodine}
The vector space $T_q$ has a basis consisting of $\{x^ay^b | a,b \in \mathbb Z\}$.
\end{lemma}

\begin{definition} \label{label}
\rm
For the algebra $T_q$, define
\begin{equation} \label{xyz}
    w_0 = x + x^{-1}, \qquad w_1 = y + y^{-1}.
\end{equation}
\end{definition}
We describe how $w_0$ and $w_1$ are related.

\begin{lemma} \label{ww}
In the algebra $T_q,$ we have
\begin{align}[{w}_0,[{w}_0,{w}_1]_q]_{q^{-1}} &= -(q^2-q^{-2})^2{w}_1, \label{abbottp} \\
[{w}_1,[{w}_1,{w}_0]_q]_{q^{-1}} &= -(q^2-q^{-2})^2{w}_0. \label{costellop}
\end{align}

\end{lemma}

\begin{proof} 
Write \eqref{abbottp} and \eqref{costellop} in terms of $x$ and $y$ using \eqref{xyz} and simplify the result using Lemma \ref{lem:3facts}.
\end{proof}

\begin{remark}
\rm
A variation of Definition \ref{label} and Lemma \ref{ww} can be found in \cite[eqn. (30) and above eqn. (32)]{16}.
\end{remark}

\begin{corollary} \label{barium}
In the algebra $T_q$, we have
\begin{align}
[w_0,[w_0,[w_0,w_1]_q]_{q^{-1}}] &= -(q^2-q^{-2})^2[w_0,w_1], \label{abbott}\\
[w_1,[w_1,[w_1,w_0]_q]_{q^{-1}}] &= -(q^2-q^{-2})^2[w_1,w_0]. \label{costello}
\end{align}
\end{corollary}

\begin{proof}
 To obtain equation \eqref{abbott}, take the commutator of $w_0$ with both sides of equation \eqref{abbottp}. To obtain equation \eqref{costello}, take the commutator of $w_1$ with both sides of equation \eqref{costellop}.
\end{proof}


\begin{proposition} \label{p}
There exists an algebra homomorphism $p : O_q \mapsto T_q$ that sends $W_0 \mapsto w_0$ and $W_1 \mapsto w_1$.
\end{proposition}

\begin{proof}
The relations in Corollary \ref{barium} are the $q$-Dolan-Grady relations.
\end{proof}

\begin{definition} \label{cerium}
\rm
For $k \in \mathbb{N}$, we define the following elements of $T_q$:
\begin{equation} \label{thxgvn}w_{-k} = p(W_{-k}), \quad w_{k+1} = p(W_{k+1}), \quad  g_{k+1} = p({G}_{k+1}), \quad \tilde{g}_{k+1}=p(\tilde{G}_{k+1}). \end{equation}
The elements in \eqref{thxgvn} are called the \textit{alternating elements} of $T_q$.
\end{definition}

\medskip

For notational convenience, we define
\begin{align}
    g_0= -\left(q-q^{-1}\right)\left(q+q^{-1}\right)^2, \qquad \qquad
    \tilde{g}_0=  -\left(q-q^{-1}\right)\left(q+q^{-1}\right)^2.  \nonumber
\end{align}

Since $O_q$ is generated by $W_0$ and $W_1$, the alternating elements of $O_q$ are polynomials in $W_0$ and $W_1$. However, only recursive formulas for these polynomials are known. The situation is different for $T_q$. We will express the alternating elements of $T_q$ in the basis for $T_q$ given in Lemma \ref{iodine}. This will provide attractive closed forms for the alternating elements of $T_q$.

\section{An automorphism and antiautomorphism of $T_q$}

In this section we define an automorphism and an antiautomorphism of $T_q$.
\begin{lemma} \label{promethium}
There exists an automorphism $\tau$ of $T_q$ that sends
$$ x \mapsto y^{-1}, \qquad y \mapsto x, \qquad x^{-1} \mapsto y, \qquad y^{-1} \mapsto x^{-1}.$$
\end{lemma}

\begin{proof}
Define \begin{equation} \label{roger} x_{\tau} = y^{-1}, \quad y_{\tau} = x, \quad (x^{-1})_\tau = y, \quad (y^{-1})_{\tau}=x^{-1}. \nonumber \end{equation}
One routinely checks that $x_{\tau}$, $y_{\tau},(x^{-1})_{\tau}$, and $(y^{-1})_{\tau}$ satisfy the defining relations in \eqref{aldo}.

Consequently, there is an algebra homomorphism $\tau:T_q \rightarrow T_q$ that sends

$$x \mapsto x_{\tau}, \qquad y \mapsto y_{\tau}, \qquad x^{-1} \mapsto (x^{-1})_{\tau}, \qquad y^{-1} \mapsto (y^{-1})_{\tau}.$$

Note that $\tau^4$ fixes $x$, $y$, $x^{-1}$, and $y^{-1}$. Hence, $\tau^4$ is the identity map, so the map $\tau$ is invertible. Thus $\tau$ is an automorphism of $T_q$.
\end{proof}





\begin{lemma} \label{samarium}
There exists an antiautomorphism $\ddagger$ of $T_q$ that sends

$$ x \mapsto x^{-1}, \qquad y \mapsto y, \qquad x^{-1} \mapsto x, \qquad y^{-1} \mapsto y^{-1}.$$
\end{lemma}

\begin{proof}
Define
$$x_{\ddagger} = x^{-1}, \qquad y_{\ddagger} = y, \qquad (x^{-1})_{\ddagger} = x, \qquad (y^{-1})_{\ddagger} = y^{-1}.$$

The opposite algebra $T^{\rm op}_q$ is explained in Section 2. By this explanation and \eqref{aldo}, the defining relations of $T_q^{\rm op}$ are

\begin{equation} \label{talbot}
xx^{-1} = 1 = x^{-1}x, \qquad yy^{-1} = 1 = y^{-1}y, \qquad yx = q^2 xy. \nonumber \end{equation}

These relations are satisfied by $x_{\ddagger}, y_{\ddagger}, x^{-1}_{\ddagger}, y^{-1}_{\ddagger}$.

Therefore, there is an algebra homomorphism $\ddagger:T_q \rightarrow T_q^{\rm op}$ that sends
$$x \mapsto x_{\ddagger}, \qquad y \mapsto y_{\ddagger}, \qquad x^{-1} \mapsto (x^{-1})_{\ddagger}, \qquad y^{-1} \mapsto (y^{-1})_{\ddagger}.$$

Note that $\ddagger^2$ fixes $x$, $y$, $x^{-1}$, and $y^{-1}$, so $\ddagger^2$ is the identity map of $T_q$. It follows that $\ddagger$ is a bijection and hence an antiautomorphism of $T_q$.

\end{proof}


\begin{lemma} \label{europium}
These diagrams commute:

\begin{center}
\begin{tikzcd}
O_q \arrow[r,"\sigma"] \arrow[d,"p"'] & O_q \arrow[d,"p"] \\
T_q \arrow[r,"\tau"'] & T_q 
\end{tikzcd}
\qquad \qquad
\begin{tikzcd}
O_q \arrow[r,"\dagger"] \arrow[d,"p"'] & O_q \arrow[d,"p"] \\
T_q \arrow[r,"\ddagger"'] & T_q 
\end{tikzcd}
\end{center}
\end{lemma}

\begin{proof}
Chase $W_0$ and $W_1$ around each diagram.
\end{proof}

\section{Some generating functions}
\indent Our goal for the paper is to write the alternating elements of $T_q$ in the basis from Lemma \ref{iodine}. We do so using generating functions. In this section we define these generating functions
and discuss their properties.

\begin{definition} \label{thallium}
\rm (See \cite[Definition 7.1]{T2}.)
For the algebra $\mathcal O_q$, define the generating functions
\begin{equation} \begin{split} \label{xthallium}
\mathcal  W^-(t) = \sum_{i=0}^{\infty} \mathcal W_{-i}t^i, \qquad & \qquad
\mathcal W^+(t) = \sum_{i=0}^{\infty} \mathcal W_{i+1}t^i, \\
\mathcal  G(t) = \sum_{i=0}^{\infty}  \mathcal G_i t^i, \qquad & \qquad
\tilde{{\mathcal G}}(t) = \sum_{i=0}^{\infty} \tilde{{\mathcal G}}_i t^i.
\end{split} \end{equation}
\end{definition}

\begin{definition} \label{biseo}
\rm (See \cite[Definition 12.1]{T2}.)
For the algebra $ O_q$, define the generating functions
\begin{equation} \begin{split} \label{xbiseo}
 W^-(t) = \sum_{i=0}^{\infty} W_{-i}t^i, \qquad & \qquad
 W^+(t) = \sum_{i=0}^{\infty} W_{i+1}t^i,  \\
 G(t) = \sum_{i=0}^{\infty}  G_i t^i, \qquad & \qquad
\tilde{{G}}(t) = \sum_{i=0}^{\infty} \tilde{{G}}_i t^i. 
\end{split} \end{equation}

\end{definition}

\begin{definition} \label{polonium}
\rm
For the algebra $T_q$, define the generating functions
\begin{equation} \begin{split} \label{xpolonium}
    w^-(t) = \sum_{i=0}^{\infty} w_{-i}t^i, \qquad & \qquad
    w^+(t) = \sum_{i=0}^{\infty} w_{i+1}t^i, \\
    g(t) = \sum_{i=0}^{\infty} g_{i}t^i, \qquad  & \qquad
    \tilde{g}(t) = \sum_{i=0}^{\infty} \tilde{g}_{i}t^i.
\end{split} \end{equation}
\end{definition}

Recall the algebra homomorphism $\gamma$ from Definition \ref{indium}. 

\begin{lemma} 
The map $\gamma$ sends the generating functions in \eqref{xthallium} to the generating functions in \eqref{xbiseo} as follows.
\begin{center}
\begin{tabular}{c | c c c c} 

 \text{argument in} $\mathcal O_q$ & $\mathcal W^-(t)$ & $\mathcal W^+(t)$ & $\mathcal G(t)$ & $\tilde{\mathcal G}(t)$ \\ 
 \hline
 \text{image in} $O_q$ & $W^-(t)$ & $W^+(t)$ & $G(t)$ & $\tilde{G}(t)$ \\

\end{tabular}
\end{center}
\end{lemma}

\begin{proof}
By Definition \ref{worc}.
\end{proof}

Recall the algebra homomorphism $p$ from Proposition \ref{p}.

\begin{lemma}
The map $p$ sends the generating functions in \eqref{xbiseo}  to the generating functions in \eqref{xpolonium} as follows.

\begin{center}
\begin{tabular}{c | c c c c} 

 \text{argument in} $O_q$ & $W^-(t)$ & $W^+(t)$ & $G(t)$ & $\tilde{G}(t)$ \\ 
 \hline
 \text{image in} $T_q$ & $w^-(t)$ & $w^+(t)$ & $g(t)$ & $\tilde{g}(t)$ \\

\end{tabular}
\end{center}

\end{lemma}

\begin{proof}
By Definition \ref{cerium}.
\end{proof}

Next, we express the relations \eqref{eq:3p1}--\eqref{eq:3p11}  in terms of generating functions. 

\begin{lemma} {\rm (See \cite[Lemma 7.2]{T2}.)} The following relations hold in the algebra $\mathcal O_q$.
\begin{align}
&
 \lbrack \mathcal W_0, \mathcal W^+(t)\rbrack= 
\lbrack \mathcal W^-(t), \mathcal W_{1}\rbrack=
t^{-1}({\mathcal{\tilde G}(t)} - \mathcal G(t))/(q+q^{-1}),
\label{eq:4p1}
\\
&
\lbrack \mathcal W_0, \mathcal G(t)\rbrack_q= 
\lbrack {\mathcal{\tilde G}(t)}, \mathcal W_{0}\rbrack_q= 
\rho  \mathcal W^-(t)-\rho t
 \mathcal W^+(t),
\label{eq:4p2}
\\
&
\lbrack \mathcal G(t), \mathcal W_{1}\rbrack_q= 
\lbrack \mathcal W_{1}, {\mathcal {\tilde G}}(t)\rbrack_q= 
\rho  \mathcal W^+(t)-\rho t
 \mathcal W^-(t),
\label{eq:4p3}
\\
&
\lbrack \mathcal W^-(s), \mathcal W^-(t)\rbrack=0,  \qquad 
\lbrack \mathcal W^+(s), \mathcal W^+(t)\rbrack= 0,
\label{eq:4p4}
\\
&
\lbrack \mathcal W^-(s), \mathcal W^+(t)\rbrack+
\lbrack \mathcal W^+(s), \mathcal W^-(t)\rbrack= 0,
\label{eq:4p5}
\\
& s
\lbrack \mathcal W^-(s), \mathcal G(t)\rbrack+t
\lbrack \mathcal G(s), \mathcal W^-(t)\rbrack= 0,
\label{eq:4p6}
\\
&
s\lbrack \mathcal W^-(s), {\mathcal {\tilde G}}(t)\rbrack+t
\lbrack {\mathcal {\tilde G}}(s), \mathcal W^-(t)\rbrack= 0,
\label{eq:4p7}
\\
&
s\lbrack \mathcal W^+(s), \mathcal G(t)\rbrack+t
\lbrack  \mathcal  G(s), \mathcal W^+(t)\rbrack= 0,
\label{eq:4p8}
\\
&
s \lbrack \mathcal W^+(s), {\mathcal {\tilde G}}(t)\rbrack+t
\lbrack {\mathcal {\tilde G}}(s), \mathcal W^+(t)\rbrack= 0,
\label{eq:4p9}
\\
&
\lbrack \mathcal G(s), \mathcal G(t)\rbrack=0,
\qquad 
\lbrack {\mathcal {\tilde G}}(s), {\mathcal {\tilde G}}(t)\rbrack= 0,
\label{eq:4p10}
\\
&
\lbrack {\mathcal {\tilde G}}(s), \mathcal G(t)\rbrack+
\lbrack \mathcal G(s), {\mathcal {\tilde G}}(t)\rbrack= 0.
\label{eq:4p11}
\end{align}
\end{lemma}

In these equations, $\rho = -(q^2-q^{-2})^2$. 
\medskip
We turn our attention to $O_q$.
\begin{lemma}

The following equations hold in the algebra $O_q$.
\begin{align}
&
 \lbrack W_0, W^+(t)\rbrack= 
\lbrack W^-(t), W_{1}\rbrack=
t^{-1}({{\tilde G}(t)} - G(t))/(q+q^{-1}),
\label{eq:6p1}
\\
&
\lbrack W_0, G(t)\rbrack_q= 
\lbrack {{\tilde G}(t)}, W_{0}\rbrack_q= 
\rho   W^-(t)-\rho t
W^+(t),
\label{eq:6p2}
\\
&
\lbrack  G(t),  W_{1}\rbrack_q= 
\lbrack  W_{1}, { {\tilde G}}(t)\rbrack_q= 
\rho   W^+(t)-\rho t
  W^-(t),
\label{eq:6p3}
\\
&
\lbrack  W^-(s),  W^-(t)\rbrack=0,  \qquad 
\lbrack  W^+(s),  W^+(t)\rbrack= 0,
\label{eq:6p4}
\\
&
\lbrack  W^-(s),  W^+(t)\rbrack+
\lbrack  W^+(s),  W^-(t)\rbrack= 0,
\label{eq:6p5}
\\
& s
\lbrack  W^-(s),  G(t)\rbrack+t
\lbrack  G(s),  W^-(t)\rbrack= 0,
\label{eq:6p6}
\\
&
s\lbrack  W^-(s), { {\tilde G}}(t)\rbrack+t
\lbrack { {\tilde G}}(s),  W^-(t)\rbrack= 0,
\label{eq:6p7}
\\
&
s\lbrack  W^+(s),  G(t)\rbrack+t
\lbrack    G(s),  W^+(t)\rbrack= 0,
\label{eq:6p8}
\\
&
s \lbrack  W^+(s), { {\tilde G}}(t)\rbrack+t
\lbrack { {\tilde G}}(s),  W^+(t)\rbrack= 0,
\label{eq:6p9}
\\
&
\lbrack  G(s),  G(t)\rbrack=0,
\qquad 
\lbrack { {\tilde G}}(s), { {\tilde G}}(t)\rbrack= 0,
\label{eq:6p10}
\\
&
\lbrack {{\tilde G}}(s),  G(t)\rbrack+
\lbrack G(s), {{\tilde G}}(t)\rbrack= 0.
\label{eq:6p11}
\end{align}

\end{lemma}
\begin{proof}
Apply the map $\gamma$ to relations \eqref{eq:4p1}--\eqref{eq:4p11}.
\end{proof}

We now turn our attention to $T_q$.

\begin{lemma}
The following equations hold in the algebra $T_q$. 
\begin{align}
&
 \lbrack w_0, w^+(t)\rbrack= 
\lbrack w^-(t), w_{1}\rbrack=
t^{-1}({{\tilde g}(t)} - g(t))/(q+q^{-1}),
\label{eq:9p1}
\\
&
\lbrack w_0, g(t)\rbrack_q= 
\lbrack {{\tilde g}(t)}, w_{0}\rbrack_q= 
\rho   w^-(t)-\rho t
w^+(t),
\label{eq:9p2}
\\
&
\lbrack  g(t),  w_{1}\rbrack_q= 
\lbrack  w_{1}, { {\tilde g}}(t)\rbrack_q= 
\rho   w^+(t)-\rho t
  w^-(t),
\label{eq:9p3}
\\
&
\lbrack  w^-(s),  w^-(t)\rbrack=0,  \qquad 
\lbrack  w^+(s),  w^+(t)\rbrack= 0,
\label{eq:9p4}
\\
&
\lbrack  w^-(s),  w^+(t)\rbrack+
\lbrack  w^+(s),  w^-(t)\rbrack= 0,
\label{eq:9p5}
\\
& s
\lbrack  w^-(s),  g(t)\rbrack+t
\lbrack  g(s),  w^-(t)\rbrack= 0,
\label{eq:9p6}
\\
&
s\lbrack  w^-(s), { {\tilde g}}(t)\rbrack+t
\lbrack { {\tilde g}}(s),  w^-(t)\rbrack= 0,
\label{eq:9p7}
\\
&
s\lbrack  w^+(s),  g(t)\rbrack+t
\lbrack    g(s),  w^+(t)\rbrack= 0,
\label{eq:9p8}
\\
&
s \lbrack  w^+(s), { {\tilde g}}(t)\rbrack+t
\lbrack { {\tilde g}}(s),  w^+(t)\rbrack= 0,
\label{eq:9p9}
\\
&
\lbrack  g(s),  g(t)\rbrack=0,
\qquad 
\lbrack { {\tilde g}}(s), { {\tilde g}}(t)\rbrack= 0,
\label{eq:9p10}
\\
&
\lbrack {{\tilde g}}(s),  g(t)\rbrack+
\lbrack g(s), {{\tilde g}}(t)\rbrack= 0.
\label{eq:9p11}
\end{align}
\end{lemma}
\begin{proof}
Apply the map $p$ to relations \eqref{eq:6p1}--\eqref{eq:6p11}.
\end{proof}

\section{The generating function $\omega(t)$}

In this section, we introduce a function that will become useful later on.

\begin{definition} \label{mendelevium}
\rm
Define the generating function

$$\omega(t) = \sum_{i=0}^{\infty} \binom{2i}{i} \left(\frac{t}{q+q^{-1}}\right)^{2i}.$$ 
\end{definition}

Note that $\omega(0)=1$.

\medskip
\noindent \textbf{Remark:} The function $\omega(t)$ is the power series expansion of 

$$\left(1-\frac{4t^2}{(q+q^{-1})^2}\right)^{-\frac{1}{2}}.$$

However, we will not make use of this fact.

\medskip

Define  \begin{equation} \label{aaron} T=\frac{q+q^{-1}}{qt+q^{-1}t^{-1}}, \qquad \qquad S = \frac{q+q^{-1}}{qt^{-1}+q^{-1}t}.\end{equation}

Following \cite{T3}, we express  $T$ and $S$ as power series:  \begin{equation} \label{ts1989} T = (q+q^{-1})\sum_{\ell=0}^{\infty} (-1)^{\ell}q^{2\ell+1}t^{2\ell+1}, \qquad S = (q+q^{-1})\sum_{\ell=0}^{\infty} (-1)^{\ell}q^{-(2\ell+1)}t^{2\ell+1}. \end{equation}

\begin{lemma} \label{55}

The following identities hold:

\begin{equation} \label{ytterbium}
\omega(T) =  \frac{1+q^2t^2}{1-q^2t^2}, \qquad \qquad \omega(S) = \frac{1+q^{-2}t^2}{1-q^{-2}t^2}.
\end{equation} 
\end{lemma}

\begin{proof}
We first verify the equation on the left in \eqref{ytterbium}. From \eqref{aaron} and Definition \ref{mendelevium}, we obtain
\begin{equation}
\omega(T) = \sum_{i=0}^{\infty} \binom{2i}{i} \left(\frac{1}{qt+q^{-1}t^{-1}}\right)^{2i}. \nonumber
\end{equation}

Recall a special case of the Newton binomial theorem \cite[Theorem 5.5.1]{combo}, which states that for any integer $a$,
$$(1+z)^a = \sum_{j=0}^{\infty}\binom{a}{j}z^j,$$
where we define

$$\binom{a}{j}=\frac{a(a-1) \cdots (a-j+1)}{j!}.$$

One readily checks that 

$$\binom{-a}{j} =  \left(-1\right)^j\binom{a+j-1}{j}.$$

We express $(1+q^2t^2)^{-1}\omega(T)$ as a power series in $t$. We have
\begin{align} 
(1+q^2t^2)^{-1}\omega(T) &= \sum_{i=0}^{\infty} \binom{2i}{i}\frac{(qt)^{2i}}{(1+q^2t^2)^{2i+1}} \nonumber \\
&= \sum_{i=0}^{\infty} \binom{2i}{i} q^{2i}t^{2i} \sum_{j=0}^{\infty} \left(-1\right)^j\binom{2i+j}{j}q^{2j}t^{2j} \nonumber \\
&= \sum_{n=0}^{\infty}\sum_{i=0}^{n}(-1)^{n-i}\binom{2i}{i}q^{2n}t^{2n}\binom{n+i}{n-i} \nonumber \\
&= \sum_{n=0}^{\infty}c_nq^{2n}t^{2n} \label{peep} \nonumber
\end{align}

where

\begin{equation} \label{cn} c_n = \sum_{i=0}^{n}\left(-1\right)^{n-i}\binom{2i}{i}\binom{n+i}{n-i}. \end{equation}

We claim that $c_n = 1$ for all $n \ge 0$. Let $n$ be given and consider the expression
\begin{equation} \label{xn} \psi(x)=\left(1-x\right)^{-n-1}\left(1-x\right)^n.\end{equation}

We will compute the coefficient of $x^n$ in two ways. First, notice that 
$$\psi(x)=\left(1-x\right)^{-n-1}(1-x)^n=(1-x)^{-1}=\sum_{i=0}^{\infty} x^i$$
so the coefficient of $x^n$ in $\psi(x)$ is equal to $1$.

\medskip
Second, we expand \eqref{xn} using the Newton binomial theorem. We have
\begin{align} 
\psi(x)=(1-x)^{-n-1}(1-x)^n &= \sum_{s=0}^{\infty}x^s\sum_{i=0}^{s}(-1)^{s+i}\binom{n+i}{i}\binom{n}{s-i}. \label{notanumber}
\end{align}
In the double sum on the right of \eqref{notanumber}, the coefficient of $x^n$ is 
\begin{align} \label{75}
\sum_{i=0}^n(-1)^{n+i}\binom{n+i}{i}\binom{n}{n-i} &= \sum_{i=0}^n(-1)^{n-i}\binom{2i}{i}\binom{n+i}{n-i}.
\end{align}

The right side of \eqref{75} matches the right side of \eqref{cn}, and therefore is equal to $c_n$. Hence, we have shown that $c_n=1$, and the claim is proven. By the claim, we have
$$(1+q^2t^2)^{-1}\omega(T)= \sum_{n=0}^{\infty}q^{2n}t^{2n}=\frac{1}{1-q^2t^2}.$$
We have verified the equation on the left in \eqref{ytterbium}. The equation on the right in \eqref{ytterbium} is similarly verified.
\end{proof}

\section{The first main result} \label{ofr}

Recall the generating functions from Definition \ref{polonium}.
In this section, we express these generating functions in the basis for $T_q$ given in Lemma \ref{iodine}. We now state our first main result.

\begin{theorem} \label{hartshorne}
In the algebra $T_q$,
\begin{align}
    w^-(t) &= \omega(t)\left(x+x^{-1}\right), \label{icoileray} \\
    w^+(t) &=
    \omega(t)\left(y+y^{-1}\right), \label{thisneedsalabel} \\
      g(t) &= \left(q^2-q^{-2}\right)\omega(t)\left(t(qx^{-1}y+qxy^{-1})-q-q^{-1}\right), \label{-1heartdove}\\
    \tilde{g}(t) &= \left(q^2-q^{-2}\right)\omega(t)\left(t(q^{-1}xy+q^{-1}x^{-1}y^{-1})-q-q^{-1}\right). \label{ialyssa}
\end{align}
\end{theorem}
Theorem \ref{hartshorne} will be proven at the end of this section.

\begin{definition} 
\rm
In the algebra $\mathcal O_q$, define

$$\mathcal Z(t) = \sum_{i=0}^{\infty} \mathcal Z_i \left(q+q^{-1}\right)^it^i,$$

where $\mathcal Z_i$ is from Definition \ref{pomegranate}.
\end{definition}
\medskip
We present two lemmas about $\mathcal Z(t)$. Recall $S$, $T$ from \eqref{aaron} and the homomorphism $\gamma$ from Definition \ref{indium}. 
\begin{lemma} \label{neptunium}
The map $\gamma$ sends $\mathcal Z(t) \mapsto (q+q^{-1})^2$.
\end{lemma}

\begin{proof}
By Lemma \ref{pittsburgh}, $\gamma(\mathcal Z_i) =0$ for $i \ge 1$. Recall from Definition \ref{pomegranate} that $\mathcal Z_0 = (q+q^{-1})^2$. The result follows.
\end{proof}

\begin{lemma}  {\rm (See \cite[Definition 8.4]{T3}.)} \label{lazor} In the algebra $\mathcal O_q$, 
\begin{equation}\begin{split} 
\mathcal Z(t) &= t^{-1}ST\mathcal W^{-}(S) \mathcal W^+(T) + tST \mathcal W^+(S) \mathcal W^-(T)  -q^2ST \mathcal W^-(S) \mathcal W^-(T) \\ & \quad -q^{-2}ST\mathcal W^+(S) \mathcal W^+(T) + (q^2-q^{-2})^{-2} \mathcal G(S) \tilde{\mathcal G} (T). \label{munoz}
\end{split}\end{equation}

\end{lemma}
\begin{proof}
Let $\mathcal{Z}^c(t)$ denote the expression on the right side of equation \eqref{munoz}. We expand $S$ and $T$ using the power series in \eqref{ts1989}. These expansions each have zero constant term, so $t^{-1}ST$ can be written as a power series in $t$. Hence $\mathcal Z^c(t)$ is a power series, which we express as
$$\mathcal{Z}^c(t) = \sum_{i=0}^{\infty} \mathcal{Z}^c_i(q+q^{-1})^it^i.$$
For $i \ge 1$, the map $\phi$ sends $1 \otimes z_i \mapsto \mathcal Z^c_i $ by \cite[Lemma 9.3]{T3}, so $\mathcal Z^c_i = \mathcal Z_i$. By \cite[Lemma 8.18]{T3}, we have $\mathcal Z_0^c = (q+q^{-1})^2 = \mathcal Z_0$.  By these observations, $\mathcal Z^c(t) = \mathcal Z(t)$ and the result follows.
\end{proof}

\begin{corollary} \label{milwaukee}
In the algebra $O_q$,
\begin{equation} \begin{split} 
\hspace{-0.20in} \left(q+q^{-1}\right)^2 &= t^{-1}ST W^{-}(S)  W^+(T) + tST  W^+(S)  W^-(T)  -q^2ST  W^-(S) W^-(T) \\ & \quad -q^{-2}ST W^+(S)  W^+(T) + \left(q^2-q^{-2}\right)^{-2} G(S) \tilde{G} (T).\label{76} \end{split} \end{equation}
\end{corollary}
\begin{proof}
Routine consequence of Lemmas \ref{neptunium} and \ref{lazor}.
\end{proof}

\begin{corollary} \label{americium}
In the algebra $T_q$,
\begin{equation} \begin{split}
\left(q+q^{-1}\right)^2 &= t^{-1}STw^{-}(S)w^+(T) + tSTw^+(S)w^-(T) -q^2STw^-(S)w^-(T) \\ 
& \quad -q^{-2}STw^+(S)w^+(T) + \left(q^2-q^{-2}\right)^{-2}g(S)\tilde{g}(T). \label{flexen}
\end{split}\end{equation}
\end{corollary}
\begin{proof}
Apply $p$ to both sides of equation \eqref{76}.
\end{proof}

\begin{proof}[Proof of Theorem \ref{hartshorne}]

By \cite[Lemma 8.22]{T3}, there is a unique set of generating functions in $\mathcal O_q$,
\begin{equation} \label{useyourillusion1} \mathcal W_c^{-}(t), \qquad \mathcal W_c^{+}(t), \qquad \mathcal G_c(t), \qquad \tilde{\mathcal G}_c(t) \end{equation}
that satisfy equations \eqref{eq:4p1}--\eqref{eq:4p11}, equation \eqref{munoz}, and the conditions
\begin{equation} \label{bethebird} \mathcal W_c^-(0) = \mathcal W_0, \qquad \mathcal W_c^+(0) = \mathcal W_1. \end{equation}
Applying $\rho \circ \gamma$ to \eqref{useyourillusion1} and \eqref{bethebird}, we see that there is a unique set of generating functions in $T_q$,
\begin{equation} \label{useyourillusion2}
w_c^{-}(t), \qquad w_c^{+}(t), \qquad  g_c(t), \qquad \tilde{g}_c(t), \nonumber \end{equation}
that satisfy equations \eqref{eq:9p1}--\eqref{eq:9p11}, equation \eqref{flexen}, and the conditions

\begin{equation} \label{countmein} w_c^-(0) = w_0, \qquad  w_c^+(0) =  w_1.\end{equation}   

We are going to display this unique solution. Going forward, we define
\begin{equation}w_c^-(t), \qquad w_c^+(t), \qquad g_c(t), \qquad \tilde{g}_c(t)\nonumber \end{equation}
to be the right-hand sides of equations $\eqref{icoileray}$--$\eqref{ialyssa}$, in that order.

Since $\omega(0)=1$, these candidates satisfy \eqref{countmein}. Therefore, it suffices to show that these candidates satisfy equations \eqref{flexen} and \eqref{eq:9p1}--\eqref{eq:9p11}. 

We begin with \eqref{flexen}. We claim that in the algebra $T_q$,
\begin{equation}\begin{split}
 \left(q+q^{-1}\right)^2 &= t^{-1}STw_c^{-}(S)w_c^+(T) + tSTw_c^+(S)w_c^-(T)  -q^2STw_c^-(S)w_c^-(T) \\ 
& \quad -q^{-2}STw_c^+(S)w_c^+(T) + \left(q^2-q^{-2}\right)^{-2}g_c(S)\tilde{g}_c(T). \label{avva}
\end{split}\end{equation}
To show this, we evaluate equations \eqref{icoileray}--\eqref{ialyssa}  using \eqref{aaron} and \eqref{ytterbium} to obtain the following equations:
\begin{align}
w_c^-(S) &=  \frac{1+q^{-2}t^2}{1-q^{-2}t^2}\left(x+x^{-1}\right), \label{eightyfive22} \\
w_c^+(S) &=  \frac{1+q^{-2}t^2}{1-q^{-2}t^2}\left(y+y^{-1}\right), \label{cplus} \\
w_c^-(T) &=  \frac{1+q^2t^2}{1-q^2t^2}\left(x+x^{-1}\right), \label{cminus}\\
w_c^+(T) &= \frac{1+q^2t^2}{1-q^2t^2}\left(y+y^{-1}\right), \label{fuqkthesystem} \\
g_c(S) &=  \frac{(q+q^{-1})^2(q-q^{-1})(1+q^{-2}t^2)}{(1-q^{-2}t^2)(qt^{-1}+q^{-1}t)}\left(qxy^{-1}+qx^{-1}y-qt^{-1}-q^{-1}t\right), \label{gc} \\
\tilde{g}_c(T) &= \frac{(q+q^{-1})^2(q-q^{-1})(1+q^{2}t^2)}{(1-q^{2}t^2)(qt+q^{-1}t^{-1})}\left(q^{-1}xy+q^{-1}x^{-1}y^{-1}-qt-q^{-1}t^{-1}\right). \label{ninety32}
\end{align}

To verify \eqref{avva}, we evaluate the expression on its right using equations \eqref{eightyfive22}--\eqref{ninety32}. We have now verified \eqref{avva}, and with it \eqref{flexen}. The equations \eqref{eq:9p1}--\eqref{eq:9p11} can be verified in a straightforward manner. The reader is welcome to see Appendix A for more details.

We have now verified that the candidate functions satisfy \eqref{eq:9p1}--\eqref{eq:9p11} and \eqref{flexen}, so the proof is complete.
\end{proof}

\section{The second main result}

In this section we use the formulas in Theorem \ref{hartshorne} to express the alternating elements of $T_q$
in the basis for $T_q$ given in Lemma \ref{iodine}.

Throughout this section, $j$ is a natural number. The following is our second main result.

\begin{theorem} \label{halfmarathon}

In the table below, we express the alternating elements of $T_q$ in the basis for $T_q$ from Lemma \ref{iodine}. For $k \in \mathbb{N}$, 

\medskip
\begin{tabular}{c |c | c}
 & Case $k=2j$ & Case $k=2j+1$ \\
 \hline
$w_{k+1}$     &  $\binom{2j}{j}\left(q+q^{-1}\right)^{-2j}\left(y+y^{-1}\right)$ & $0$ \\
$w_{-k}$ & $\binom{2j}{j}\left(q+q^{-1}\right)^{-2j}\left(x+x^{-1}\right)$ & $0$ \\
$g_k$ & $-\binom{2j}{j}\left(q-q^{-1}\right)\left(q+q^{-1}\right)^{2-2j}$ & $\binom{2j}{j}q\left(q-q^{-1}\right)\left(q+q^{-1}\right)^{1-2j}\left(xy^{-1}+x^{-1}y\right)$ \\
$\tilde{g}_k$ & $-\binom{2j}{j}\left(q-q^{-1}\right)\left(q+q^{-1}\right)^{2-2j}$ & $\binom{2j}{j}q^{-1}\left(q-q^{-1}\right)\left(q+q^{-1}\right)^{1-2j}\left(xy+x^{-1}y^{-1}\right)$ \\
\end{tabular}

\end{theorem}

\begin{proof}
We begin with $\tilde{g}_k$. Expanding $\omega(t)$ in the right side of \eqref{ialyssa}, we have
\begin{equation} \label{x2} \tilde{g}(t) = \left(q^2-q^{-2}\right)\left(t(q^{-1}xy+q^{-1}x^{-1}y^{-1})- (q+q^{-1})\right)\sum_{i=0}^{\infty} \binom{2i}{i}\frac{t^{2i}}{(q+q^{-1})^{2i}}. \end{equation}

We examine the coefficient of $t^k$ in \eqref{x2} and obtain the following equations:
\begin{align} 
\tilde{g}_{k} = -\binom{2j}{j}\left(q-q^{-1}\right)\left(q+q^{-1}\right)^{2-2j} \hspace{1.385in} & \text{ for } k=2j , \label{1} \\
\tilde{g}_{k} = \binom{2j}{j} \left(q-q^{-1}\right)\left(q+q^{-1}\right)^{1-2j} \left(q^{-1}xy+q^{-1}x^{-1}y^{-1}\right) & \text{ for } k=2j+1. \label{2}
 \end{align}

Expanding $\omega(t)$ in the right side of  \eqref{-1heartdove}, we have
\begin{equation} \label{y2} g(t) = \left(q^2-q^{-2}\right)\left(t(qxy^{-1}+qx^{-1}y)-(q+q^{-1})\right)\sum_{i=0}^{\infty} \binom{2i}{i}\frac{t^{2i}}{(q+q^{-1})^{2i}}. \end{equation}

We examine the coefficient of $t^k$ in \eqref{y2} and obtain the following equations:
\begin{align}
g_k = -\binom{2j}{j}\left(q-q^{-1}\right)\left(q+q^{-1}\right)^{2-2j} \hspace{1.185in} & \text{ for } k=2j, \\
g_k = \binom{2j}{j} \left(q-q^{-1}\right)\left(q+q^{-1}\right)^{1-2j} \left(qxy^{-1}+qx^{-1}y\right) \hspace{.10295in} &\text{ for } k=2j+1.
\end{align}

Expanding $\omega(t)$ in the right side of $\eqref{thisneedsalabel}$, we have
\begin{equation} \label{x3} w^+(t) = \left(y+y^{-1}\right)\sum_{i=0}^{\infty}\binom{2i}{i}\frac{t^{2i}}{(q+q^{-1})^{2i}}.\end{equation}

We examine the coefficient of $t^k$ in \eqref{x3} and obtain the following equations:
\begin{align}
w_{k+1} = \binom{2j}{j}\left(q+q^{-1}\right)^{-2j}\left(y+y^{-1}\right) \hspace{1.01in} &\text{ for } k=2j, \\
w_{k+1} = 0 \hspace{2.95in} &\text{ for } k=2j+1.
\end{align}

Expanding $\omega(t)$ in the right side of $\eqref{icoileray}$, we have
\begin{equation} \label{y3} w^-(t) = \left(x+x^{-1}\right)\sum_{i=0}^{\infty}\binom{2i}{i}\frac{t^{2i}}{(q+q^{-1})^{2i}}.\end{equation}

We examine the coefficient of $t^k$ in \eqref{y3} and obtain the following equations:
\begin{align}
w_{-k} = \binom{2j}{j}\left(q+q^{-1}\right)^{-2j}\left(x+x^{-1}\right) \hspace{1.01in} &\text{ for } k=2j, \\
w_{-k} = 0 \hspace{2.962in} &\text{ for } k=2j+1.
\end{align}
\end{proof}
\medskip


We end this section with a comment. From the table given in Theorem \ref{halfmarathon}, we see that each alternating element of $T_q$ is a linear combination of the following nine elements:
$$1, \quad x, \quad y, \quad x^{-1}, \quad y^{-1}, \quad xy, \quad xy^{-1}, \quad x^{-1}y, \quad x^{-1}y^{-1}.$$

The above elements are linearly independent. Therefore,
the alternating elements of $T_q$ are contained in a 9-dimensional
subspace of $T_q$.


\newpage

\section*{Appendix A: Details from the proof of Theorem \ref{hartshorne}}
This appendix clarifies some details regarding the proof of Theorem \ref{hartshorne}. We seek to show that the generating functions in \eqref{icoileray}--\eqref{ialyssa} satisfy relations \eqref{eq:9p1}--\eqref{eq:9p11}. 

To simplify our arguments, we use the fact that the automorphism $\tau$ of $T_q$ sends
$$w^+(t) \mapsto w^-(t), \qquad w^-(t) \mapsto w^+(t), \qquad g(t) \mapsto \tilde{g}(t), \qquad \tilde{g}(t) \mapsto g(t).$$
We consider \eqref{eq:9p1}--\eqref{eq:9p11} individually.

\begin{itemize}
\item We show that $ \lbrack w_0, w^+(t)\rbrack= 
\lbrack w^-(t), w_{1} \rbrack.$

By Definition \ref{label}, we have \begin{equation} \begin{split}[w_0,w^+(t)] &= [w_0,\omega(t)(y+y^{-1})]=\omega(t)[x+x^{-1},y+y^{-1}] \\ &= [\omega(t)(x+x^{-1}),w_1]=[w^-(t),w_1]. \nonumber \end{split} \end{equation}

\item We show that $\lbrack w^-(t), w_{1}\rbrack=
\frac{{{\tilde g}(t)} - g(t)}{t(q+q^{-1})}.$

We have
\begin{align} t^{-1}\frac{\tilde{g}(t)-g(t)}{q+q^{-1}} &= t^{-1}\frac{(q^2-q^{-2})t\omega(t)(q^{-1}xy+q^{-1}x^{-1}y^{-1}-qxy^{-1}-qx^{-1}y)}{q+q^{-1}} \nonumber \\
&= (q-q^{-1})\omega(t)(q^{-1}xy+q^{-1}x^{-1}y^{-1}-qxy^{-1}-qx^{-1}y) \nonumber \\
&= \omega(t)[w_0,w_1]  \nonumber \\ &= [w^-(t),w_1]. \nonumber
\end{align}

\item We show that $\lbrack w_0, g(t)\rbrack_q= 
\rho   w^-(t)-\rho t
w^+(t).$

We have
\begin{align}
[w_0,g(t)]_q &= qw_0g(t)-q^{-1}g(t)w_0 \nonumber \\ &= \left(q^2-q^{-2}\right)\omega(t)\left(q(x+x^{-1})(tqxy^{-1}+tqx^{-1}y-q-q^{-1}) \right.\nonumber \\ & \left. \quad -q^{-1}(tqxy^{-1}+tqx^{-1}y-q-q^{-1})(x+x^{-1})\right) \nonumber \\
&=\left(q^2-q^{-2}\right)\omega(t)\left((tq^2y+tq^2x^{-2}y+tq^2y^{-1}+tq^2x^2y^{-1} \right. \nonumber \\ & \quad -q^2x-q^2x^{-1}-x-x^{-1}) - \left.(tq^{-2}y+tq^{-2}y^{-1} \right. \nonumber \\ & \quad \left. +  tq^2x^{-2}y+tq^2x^2y^{-1}-x-x^{-1}-q^{-2}x-q^{-2}x^{-1})\right) \nonumber \\
&= \left(q^2-q^{-2}\right)\omega(t)\left(q^2-q^{-2}\right)\left(ty+ty^{-1}-x-x^{-1}\right) \nonumber \\
&= -\left(q^2-q^{-2}\right)^2 \left(w^-(t)-tw^+(t)\right) \nonumber \\ &= \rho w^-(t) - \rho t w^+(t).  \nonumber 
\end{align}

\item We show that  $ 
\lbrack  w_{1}, { {\tilde g}}(t)\rbrack_q= 
\rho   w^+(t)-\rho t
  w^-(t).$

Apply the $\tau$ automorphism to the above equation.

\item We show that $\lbrack {{\tilde g}(t)}, w_{0}\rbrack_q= 
\rho   w^-(t)-\rho t
w^+(t).$

We have
\begin{align}
[\tilde{g}(t),w_0]_q &= q\tilde{g}(t)w_0-q^{-1}w_0\tilde{g}(t) \nonumber \\ &= \left(q^2-q^{-2}\right)\omega(t)\left(q(tq^{-1}xy+tq^{-1}x^{-1}y^{-1}-q-q^{-1})(x+x^{-1})\right.\nonumber \\ & \left. \quad -q^{-1}(x+x^{-1})(tq^{-1}xy+tq^{-1}x^{-1}y^{-1}-q-q^{-1})\right) \nonumber \\
&=\left(q^2-q^{-2}\right)\omega(t)\left((tq^2y+tq^{-2}x^{2}y+tq^2y^{-1}+tq^{-2}x^{-2}y^{-1} \right.\nonumber \\ & \left. \quad -q^2x-q^2x^{-1}-x-x^{-1} )-(tq^{-2}y+tq^{-2}y^{-1} \right. \nonumber \\ & \left. \quad + tq^{-2}x^{-2}y^{-1}+tq^{-2}x^2y-x-x^{-1}-q^{-2}x-q^{-2}x^{-1})\right) \nonumber \\
&= \left(q^2-q^{-2}\right)\omega(t)\left(q^2-q^{-2}\right)\left(ty+ty^{-1}-x-x^{-1}\right) \nonumber \\
&= -\left(q^2-q^{-2}\right)^2 \left(w^-(t)-tw^+(t)\right) \nonumber \\ &= \rho w^-(t) - \rho t w^+(t).  \nonumber
\end{align}

\item We show that  $\lbrack  g(t),  w_{1}\rbrack_q= 
\rho   w^+(t)-\rho t
  w^-(t).$

Apply the $\tau$ automorphism to the above equation.

\item We show that 
$\lbrack  w^-(s),  w^-(t)\rbrack=0.$

We have $$[w^-(s),w^-(t)] = \omega(s)\omega(t)[x+x^{-1},x+x^{-1}]=0.$$ 

\item We show that
$\lbrack  w^+(s),  w^+(t)\rbrack= 0.$ 

We have $$[w^+(s),w^+(t)] = \omega(s)\omega(t)[y+y^{-1},y+y^{-1}]=0.$$

\item We show that $\lbrack  w^-(s),  w^+(t)\rbrack+
\lbrack  w^+(s),  w^-(t)\rbrack= 0.$

We have \begin{align}[w^-(s),w^+(t)] = \omega(s)\omega(t)[x+x^{-1},y+y^{-1}] &= -\omega(s)\omega(t)[y+y^{-1},x+x^{-1}] \nonumber \\ &=-[w^+(s),w^-(t)]. \nonumber\end{align}

\item We show that 
$s\lbrack  w^-(s),  g(t)\rbrack+t
\lbrack  g(s),  w^-(t)\rbrack= 0.$

We have
\begin{align} s[w^-(s),g(t)] &= (q^2-q^{-2})s\omega(s)\omega(t)\left[x+x^{-1},t(qxy^{-1}+qx^{-1}y)-q-q^{-1}\right] \nonumber \\ &= (q^2-q^{-2})st\omega(s)\omega(t)[x+x^{-1},qxy^{-1}+qx^{-1}y]  \nonumber \\ &= -(q^2-q^{-2})t\omega(t)\omega(s)\left[s(qxy^{-1}+qx^{-1}y)-q-q^{-1},x+x^{-1}\right]  \nonumber \\ &= -t[g(s),w^-(t)].  \nonumber\end{align}

\item We show that 
$s\lbrack  w^+(s),  \tilde{g}(t)\rbrack+t
\lbrack  \tilde{g}(s),  w^+(t)\rbrack= 0.$

Apply $\tau$ to each side of the previous equation.

\item We show that 
$s\lbrack  w^-(s),  \tilde{g}(t)\rbrack+t
\lbrack  \tilde{g}(s),  w^-(t)\rbrack= 0.$

We have
\begin{align} s[w^-(s),\tilde{g}(t)] &= (q^2-q^{-2})s\omega(s)\omega(t)\left[x+x^{-1},t(q^{-1}xy+q^{-1}x^{-1}y^{-1})-q-q^{-1}\right]  \nonumber \\ &= (q^2-q^{-2})st\omega(s)\omega(t)[x+x^{-1},q^{-1}xy+q^{-1}x^{-1}y^{-1}]  \nonumber \\ &= -(q^2-q^{-2})t\omega(t)\omega(s)[s(q^{-1}xy+q^{-1}x^{-1}y^{-1})-q-q^{-1},x+x^{-1}]  \nonumber \\ &= -t[\tilde{g}(s),w^-(t)].  \nonumber \end{align}

\item We show that 
$s\lbrack  w^+(s),  {g}(t)\rbrack+t
\lbrack  {g}(s),  w^+(t)\rbrack= 0.$
Apply $\tau$ to each side of the previous equation.

\item We show that  $[g(s),g(t)]=0.$

We have \begin{align}
[g(s),g(t)] &= (q^2-q^{-2})^2\omega(s)\omega(t)\left[s(qxy^{-1}+qx^{-1}y)-q-q^{-1}, \right. \nonumber \\ &\left. \quad t(qxy^{-1}+qx^{-1}y)-q-q^{-1}\right] \nonumber \\
&= (q^2-q^{-2})^2st\omega(s)\omega(t)[qxy^{-1}+qx^{-1}y,qxy^{-1}+qx^{-1}y] \nonumber \\ &= 0. \nonumber 
\end{align}

\item We show that  $[\tilde{g}(s),\tilde{g}(t)]=0.$

Apply $\tau$ to each side of the previous equation.

\item We show that  $[\tilde{g}(s),g(t)]+[g(s),\tilde{g}(t)]=0.$

We have
\begin{align}
[\tilde{g}(s),g(t)] &= (q^2-q^{-2})^2\omega(s)\omega(t)\left[s(q^{-1}xy+q^{-1}x^{-1}y^{-1})-q-q^{-1}, \right.\nonumber \\ & \left.  \quad t(qxy^{-1}+qx^{-1}y)-q-q^{-1} \right] \nonumber \\
&=(q^2-q^{-2})^2st\omega(s)\omega(t)[q^{-1}xy+q^{-1}x^{-1}y^{-1},qxy^{-1}+qx^{-1}y] \nonumber \\
&=-(q^2-q^{-2})^2\omega(s)\omega(t)\left[s(qxy^{-1}+qx^{-1}y)-q-q^{-1}, \right.\nonumber \\ & \left. 
 \quad t(q^{-1}xy+q^{-1}x^{-1}y^{-1})-q-q^{-1}\right] \nonumber \\
&=-[g(s),\tilde{g}(t)]. \nonumber
\end{align}

We have now verified the equations \eqref{eq:9p1}--\eqref{eq:9p11}, so the argument is complete.
\end{itemize}

\section*{Acknowledgement}
The author is presently a graduate student at the University of Wisconsin--Madison.
He is deeply grateful to his research advisor, Professor Paul Terwilliger, for many of the resources used and for insightful comments to improve the clarity and flow of this paper. 

\noindent Owen Goff \hfil\break
\noindent Department of Mathematics \hfil\break
\noindent University of Wisconsin \hfil\break
\noindent 480 Lincoln Drive \hfil\break
\noindent Madison, WI 53706-1388 USA \hfil\break
\noindent Email: {\tt ogoff@wisc.edu }\hfil\break
\end{document}